%% file: main.tex
\documentclass[12pt]{amsart}
\usepackage{amsmath, amsthm, amssymb}
\usepackage[top=1.25in, bottom=1.25in, left=1.0in, right=1.0in]{geometry}

\allowdisplaybreaks
\usepackage{tkz-graph}
\usetikzlibrary{arrows}
\usetikzlibrary{shapes}
\usepackage[position=bottom]{subfig}

\makeatletter
\newtheorem*{rep@theorem}{\rep@title}
\newcommand{\newreptheorem}[2]{
\newenvironment{rep#1}[1]{
 \def\rep@title{#2 \ref{##1}}
 \begin{rep@theorem}}
 {\end{rep@theorem}}}
\makeatother

\theoremstyle{plain}
\newtheorem{thm}{Theorem}[section]
\newreptheorem{thm}{Theorem}

\newreptheorem{prop}{Proposition}

\newreptheorem{lem}{Lemma}

\newreptheorem{conjecture}{Conjecture}
\newtheorem{cor}[thm]{Corollary}
\newreptheorem{cor}{Corollary}

\theoremstyle{definition}
\newtheorem{defn}{Definition}
\theoremstyle{remark}

\newtheorem*{question}{Question}

\newcommand{\fancy}[1]{\mathcal{#1}}

\newcommand{\IN}{\mathbb{N}}

\newcommand{\set}[1]{\left\{ #1 \right\}}
\newcommand{\setb}[3]{\left\{ #1 \in #2 \mid #3 \right\}}
\newcommand{\setbs}[2]{\left\{ #1 \mid #2 \right\}}
\newcommand{\card}[1]{\left|#1\right|}
\newcommand{\size}[1]{\left\Vert#1\right\Vert}
\newcommand{\ceil}[1]{\left\lceil#1\right\rceil}
\newcommand{\floor}[1]{\left\lfloor#1\right\rfloor}

\newcommand{\irange}[1]{\left[#1\right]}
\newcommand{\join}[2]{#1 \mbox{\hspace{2 pt}$\ast$\hspace{2 pt}} #2}

\newcommand{\parens}[1]{\left( #1 \right)}
\newcommand{\brackets}[1]{\left[ #1 \right]}
\newcommand{\DefinedAs}{\mathrel{\mathop:}=}

\newcommand{\mov}[2]{#1^{#2}}
\newcommand{\wt}[1]{w\parens{#1}}

\renewcommand{\vec}[1]{\mathbf{#1}}

\title{Partitioning and coloring with degree constraints}
\author{Landon Rabern}

\begin{document}
	\begin{abstract}
We prove that if $G$ is a vertex critical graph with $\chi(G) \geq \Delta(G) + 1 - p \geq 4$
	for some $p \in \IN$ and $\omega(\fancy{H}(G)) \leq \frac{\chi(G) + 1}{p + 1} - 2$,
	then $G = K_{\chi(G)}$ or $G = O_5$.  Here $\fancy{H}(G)$ is the subgraph of $G$ induced on the vertices of degree at least $\chi(G)$.  This simplifies and improves the results in the paper of Kostochka, Rabern and Stiebitz \cite{krs_one}.
	\end{abstract}
	\maketitle

	\section{Introduction}
Our notation follows Diestel \cite{Diestel} unless otherwise specified.  
The natural numbers include zero; that is, $\IN \DefinedAs \set{0, 1, 2, 3, \ldots}$.  
We also use the shorthand $\irange{k} \DefinedAs \set{1, 2, \ldots, k}$.  
The complete graph on $t$ vertices is indicated by $K_t$ and the edgeless graph
on $t$ vertices by $E_t$.  A vertex $v \in V(G)$ is called \emph{universal} in $G$ if it is adjacent to every other vertex of $G$. We write $\fancy{H}(G)$ for the subgraph of $G$ induced on the vertices of degree at least $\chi(G)$.

The classical theorem of Brooks \cite{brooks1941colouring} gives the necessary and sufficient conditions for a graph to be $\Delta$-colorable.

\begin{thm}[Brooks \cite{brooks1941colouring} 1941]
If $G$ is a graph with $\chi(G) \geq \Delta(G) + 1 \geq 4$ then $G$ contains $K_{\chi(G)}$.
\end{thm}

In \cite{kierstead2009ore} Kierstead and Kostochka investigated the same question with the 
Ore-degree $\theta$ in place of $\Delta$.

\begin{defn}
The \emph{Ore-degree} of an edge $xy$ in a graph $G$ is $\theta(xy) \DefinedAs d(x) + d(y)$.  The \emph{Ore-degree} of a graph $G$ is $\theta(G) \DefinedAs \max_{xy \in E(G)}\theta(xy)$.
\end{defn}

\begin{thm}[Kierstead and Kostochka \cite{kierstead2009ore} 2010]
If $G$ is a graph with $\chi(G) \geq \floor{\frac{\theta(G)}{2}} + 1 \geq 7$ then $G$ contains $K_{\chi(G)}$.
\end{thm}

This statement about Ore-degree is equivalent to the following statement about vertex critical graphs.

\begin{thm}[Kierstead and Kostochka \cite{kierstead2009ore} 2010]
The only vertex critical graph $G$ with $\chi(G) \geq \Delta(G) \geq 7$ such that $\fancy{H}(G)$ is edgeless is $K_{\chi(G)}$.
\end{thm}

In \cite{rabern2010a}, we improved the $7$ to $6$ by proving the following generalization.

\begin{thm}[Rabern 2012 \cite{rabern2010a}]
The only vertex critical graph $G$ with $\chi(G) \geq \Delta(G) \geq 6$ and $\omega(\fancy{H}(G)) \leq \floor{\frac{\Delta(G)}{2}} - 2$ is $K_{\chi(G)}$.
\end{thm}

This result and those in \cite{rabern2010b} were improved by Kostochka, Rabern and Stiebitz in \cite{krs_one}.  In particular, the following was proved.

\begin{thm}[Kostochka, Rabern and Stiebitz \cite{krs_one} 2012]
The only vertex critical graphs $G$ with $\chi(G) \geq \Delta(G) \geq 5$ such that $\fancy{H}(G)$ is edgeless are $K_{\chi(G)}$ and $O_5$.
\end{thm}

\input{reducer}

Here $O_n$ is the graph formed from the disjoint union of $K_n - xy$ and
$K_{n-1}$ by joining $\floor{\frac{n-1}{2}}$ vertices of the $K_{n-1}$ to $x$
and the other $\ceil{\frac{n-1}{2}}$ vertices of the $K_{n-1}$ to $y$ (see
Figure \ref{fig:reducer}). In this paper we prove a result which implies all of the results in \cite{krs_one}.  The proof replaces an algorithm of Mozhan \cite{mozhan1983} with the original, more general, algorithm of Catlin \cite{CatlinAnotherBound} on which it is based. This allows for a considerable simplification.  Moreover, we prove two preliminary partitioning results that are of independent interest.  All coloring results follow from the first of these, the second is a generalization of a lemma due to Borodin \cite{borodin1976decomposition} (and independently Bollob\'as and Manvel \cite{bollobasManvel}) about partitioning a graph into degenerate subgraphs.  The following is the main coloring result in this paper.

   \begin{repcor}{ThirdColoring}
	Let $G$ be a vertex critical graph with $\chi(G) \geq \Delta(G) + 1 - p \geq 4$
	for some $p \in \IN$.  If $\omega(\fancy{H}(G)) \leq \frac{\chi(G) + 1}{p + 1} - 2$,
	then $G = K_{\chi(G)}$ or $G = O_5$.
	\end{repcor}

	\section{Partitioning}
	An \emph{ordered partition} of a graph $G$ is a sequence $\parens{V_1, V_2,
	\ldots, V_k}$ where the $V_i$ are pairwise disjoint and cover $V(G)$.  Note that we allow the $V_i$ to be
empty.  When there is no possibility of ambiguity, we call such a sequence a
\emph{partition}.	For a vector $\vec{r} \in \IN^k$ we take the
coordinate labeling $\vec{r} = \parens{r_1, r_2, \ldots, r_k}$ as convention. 
Define the \emph{weight} of a vector $\vec{r} \in \IN^k$ as $\wt{\vec{r}} \DefinedAs \sum_{i \in \irange{k}} r_i$.   
Let $G$ be a graph. An \emph{$\vec{r}$-partition} of $G$ is an ordered partition
$P \DefinedAs \parens{V_1, \ldots, V_k}$ of $V(G)$ minimizing \[f(P) \DefinedAs \sum_{i \in \irange{k}} \parens{\size{G[V_i]} - r_i\card{V_i}}.\]

It is a fundamental result of Lov\'asz \cite{lovasz1966decomposition} that if $P \DefinedAs \parens{V_1, \ldots, V_k}$ is an $\vec{r}$-partition of $G$ with $\wt{\vec{r}} \geq \Delta(G) + 1 - k$, then $\Delta(G[V_i]) \leq r_i$ for each $i \in \irange{k}$.  The proof is simple: if there is a vertex in a part violating the condition, then there is some part it can be moved to that decreases $f(P)$.  As Catlin \cite{CatlinAnotherBound} showed, with the stronger condition $\wt{\vec{r}} \geq \Delta(G) + 2 - k$, a vertex of degree $r_i$ in $G[V_i]$ can always be moved to some other part while maintaining $f(P)$.  Since $G$ is finite, a well-chosen sequence of such moves must always wrap back on itself.  Many authors, including Catlin \cite{CatlinAnotherBound}, Bollob\'as and Manvel \cite{bollobasManvel} and Mozhan \cite{mozhan1983} have used such techniques to prove coloring results. We generalize these techniques by taking into account the degree in $G$ of the vertex to be moved---a vertex of degree less than the maximum needs a weaker condition on $\wt{\vec{r}}$ to be moved.

For $x \in V(G)$ and $D \subseteq V(G)$ we use the notation $N_D(x) \DefinedAs N(x) \cap D$ and $d_D(x) \DefinedAs \card{N_D(x)}$. Let $\fancy{C}(G)$ be the components of $G$ and $c(G) \DefinedAs \card{\fancy{C}(G)}$. For an induced subgraph $H$ of $G$, define $\delta_G(H) \DefinedAs \min_{v \in V(H)} d_G(v)$.  We also need the following notion of a movable subgraph.
		
		\begin{defn}
			Let $G$ be a graph and $H$ an induced subgraph of $G$.  For $d \in \IN$, the
			\emph{$d$-movable subgraph} of $H$ with respect to $G$ is the subgraph
			$\mov{H}{d}$ of $G$ induced on \[\setb{v}{V(H)}{d_G(v) = d \text{ and } H-v
			\text{ is connected}}.\]
		\end{defn}

We prove two partition lemmas of similar form.  
All of our coloring results will follow from the first lemma, the second lemma is a 
degeneracy result from which Borodin's result in \cite{borodin1976decomposition} follows.  
For unification purposes, define a \emph{$t$-obstruction} as an odd cycle when
$t=2$ and a $K_{t + 1}$ when $t \geq 3$.

      \begin{thm}\label{PartitionTheorem}
			Let $G$ be a graph, $k,d \in \IN$ with $k \geq 2$ and $\vec{r} \in \IN_{\geq 2}^k$.  If $\wt{\vec{r}} \geq \max\set{\Delta(G) + 1 - k, d}$, then at least one of the following holds:
			\begin{enumerate}
			  \item $\wt{\vec{r}} = d$ and $G$ contains an induced subgraph $Q$ with $\card{Q} = d+1$ which can be partitioned into $k$ cliques $F_1, \ldots, F_k$ where 
					\begin{enumerate}
					\item $\card{F_1} = r_1 + 1$, $\card{F_i} = r_i$ for $i \geq 2$,
					\item $\card{F_1^d} \geq 2$, $\card{F_i^d} \geq 1$ for $i \geq 2$,
					\item for $i \in \irange{k}$, each $v \in V(F_i^d)$ is universal in $Q$;
					\end{enumerate}
			  \item there exists an $\vec{r}$-partition $P \DefinedAs \parens{V_1, \ldots, V_k}$ of 	
$G$ such that if $C$ is an $r_i$-obstruction in $G[V_i]$, then $\delta_G(C) \geq d$ and
			  $\mov{C}{d}$ is edgeless.
			\end{enumerate}
		\end{thm}
\begin{proof}
         For $i \in \irange{k}$, call a connected graph $C$
			\emph{$i$-bad} if $C$ is an $r_i$-obstruction such that $\mov{C}{d}$ has an edge. 
         For a graph $H$ and $i
			\in \irange{k}$, let $b_i(H)$ be the number of $i$-bad components of $H$.
			For an $\vec{r}$-partition $P \DefinedAs \parens{V_1, \ldots, V_k}$ of $G$ let
			\[b(P) \DefinedAs \sum_{i \in \irange{k}} b_i(G[V_i]).\]
			
			\noindent Let $P \DefinedAs \parens{V_1, \ldots, V_k}$ be an $\vec{r}$-partition of
			$V(G)$ minimizing $b(P)$.

			Let $i \in \irange{k}$ and $x \in V_i$ with $d_{V_i}(x) \geq r_i$.  Suppose $d_G(x) = d$.
			Then, since $\wt{\vec{r}} \geq d$, for every $j \neq i$ we have $d_{V_j}(x) \leq
			r_j$. Moving $x$ from $V_i$ to $V_j$ gives a new partition $P^*$ with $f(P^*)
			\leq f(P)$. Note that if $d_{G}(x) < d$ we would have $f(P^*) < f(P)$
			contradicting the minimality of $P$.

			Suppose $b(P) > 0$.  By symmetry, we may assume that there is a
			$1$-bad component $A_1$ of $G[V_{1}]$. Put $P_1 \DefinedAs P$ and $V_{1,i}
			\DefinedAs V_i$ for $i \in \irange{k}$. Since $A_1$ is $1$-bad we have $x_1
			\in V(\mov{A_1}{d})$ which has a neighbor in $V(\mov{A_1}{d})$. By the above we
			can move $x_1$ from $V_{1, 1}$ to $V_{1, 2}$ to get a new partition $P_2
			\DefinedAs \parens{V_{2, 1}, V_{2,2}, \ldots, V_{2,k}}$ where $f(P_2) = f(P_1)$.  
         Since removing $x_1$ from $A_1$ decreased $b_{1}(G[V_{1}])$, minimality of
			$b(P_1)$ implies that $x_1$ is in a $2$-bad component $A_2$ in $V_{2,2}$.			
			Now, we may choose $x_2 \in
			V(\mov{A_2}{d}) - \set{x_1}$ having a neighbor in $\mov{A_2}{d}$ and move
			$x_2$ from $V_{2, 2}$ to $V_{2, 1}$ to get a new partition $P_3
			\DefinedAs \parens{V_{3, 1}, V_{3,2}, \ldots, V_{3,k}}$ where $f(P_3) =
			f(P_1)$.
						
			Continue on this way to construct sequences $A_1, A_2, \ldots$, $P_1, P_2, P_3, \ldots$ and $x_1, x_2, \ldots$.  
			Since $G$ is finite, at some point we will need to reuse a leftover component; that is, 
			there is a smallest $t$ such that $A_{t + 1} - x_t = A_s - x_s$ for some $s <
			t$.  Let $j \in \irange{2}$ be such that in $V(A_s) \subseteq V_{s, j}$. 
			Then $V(A_t) \subseteq V_{t, 3-j}$.  Note that, since $A_s$ is $r_j$-regular,
			$N(x_t) \cap V(A_s - x_s) = N(x_s) \cap V(A_s - x_s)$.
			
			We claim that $s = 1$, $t = 2$, both $A_s$ and $A_t$ are complete,
			$\mov{A_s}{d}$ is joined to $A_t - x_{t-1}$ and $\mov{A_t}{d}$ is joined to $A_s - x_s$.
			
			Put $X \DefinedAs N(x_s) \cap V(\mov{A_s}{d})$.  Since $x_s$ witnesses the
			$j$-badness of $A_s$, $\card{X} \geq 1$. Pick $z \in X$.
			In $P_s$, move $z$ to $V_{s, 3-j}$ to get a new partition $P^\gamma \DefinedAs \parens{V_{\gamma,
			1}, V_{\gamma, 2}, \ldots, V_{\gamma, k}}$. Then $z$ must create an
			$r_{3-j}$-obstruction with $A_t - x_{t-1}$ in $V_{\gamma, 3-j}$ since
			$z$ is adjacent to $x_t$.  
			In particular, $N(z) \cap V(A_t - x_{t-1}) = N(x_{t-1}) \cap V(A_t -
			x_{t-1})$. Since $z$ is adjacent to $x_t$, so is $x_{t-1}$. 
			
			In $P^\gamma$, move $x_t$ to $V_{\gamma, j}$ to
			get a new partition $P^{\gamma*} \DefinedAs \parens{V_{\gamma*, 1},
			V_{\gamma*, 2}, \ldots, V_{\gamma*, k}}$. Then $x_t$ must create an
			$r_{j}$-obstruction with $A_s - z$ in $V_{\gamma*, j}$.  In
			particular, $N(z) \cap V(A_s - z) = N(x_t) \cap V(A_s - z)$.  Thus $x_s$
			is adjacent to $x_t$ and we have $N[z] \cap V(A_s) = N[x_s] \cap
			V(A_s)$.  Thus, if $A_s$ is an odd cycle, it must be a triangle.  
			Hence $A_s$ is complete.  Also, since $x_s$ is adjacent to $x_t$, using $x_s$
			in place of $z$ in the previous paragraph, we conclude that $\mov{A_s}{d}$
			is joined to $N(x_{t-1}) \cap V(A_t - x_{t-1})$ and $x_s = x_{t-1}$.  
			
			Suppose $s > 1$.  Then $x_{s-1}$ is joined to $N(x_{t-1}) \cap V(A_t -
			x_{t-1})$ and hence $A_t - x_{t-1} = A_{s - 1} - x_{s-1}$ violating
			minimality of $t$.  Whence, $s = 1$.  
				
			In $P_{t+1}$, move $z$ to $V_{t+1, 3-j}$ to
			get a new partition $P^{\beta} \DefinedAs \parens{V_{\beta, 1},
			V_{\beta, 2}, \ldots, V_{\beta, k}}$. Then $z$ must create an
			$r_{3-j}$-obstruction with $A_t - x_t$ in $V_{\beta, 3-j}$.  In
			particular, $N(z) \cap V(A_t - x_t) = N(x_t) \cap V(A_t - x_t)$.  Since $N(z)
			\cap V(A_t - x_{t-1}) = N(x_{t-1}) \cap V(A_t - x_{t-1})$, we have
			$N[x_{t-1}] \cap V(A_t) = N(z) \cap V(A_t) = N[x_t] \cap V(A_t)$.  
			Thus $A_t$ is complete.  Thus $\mov{A_s}{d}$ is joined to $A_t - x_{t-1}$.  To see that 	
			$\mov{A_t}{d}$ is joined to $A_s - x_s$, consider $P_t$ and moving any vertex in $\mov{A_t}{d}$ to $V_{t, j}$.
						
			Therefore $s = 1$, $t = 2$, both $A_s$ and $A_t$ are complete,
			$\mov{A_s}{d}$ is joined to $A_t - x_{t-1}$ and $\mov{A_t}{d}$ is joined to
			$A_s - x_s$.  But we can play the same game with $V_1$ and $V_i$ for any
			$3 \leq i \leq k$ as well.  Let $B_1 \DefinedAs A_1$, $B_2 \DefinedAs A_2$
			and for $i \geq 3$, let $B_i$ be the $r_i$-obstruction made by moving
			$x_1$ into $V_i$.  Then $B_i$ is complete for each $i \in \irange{k}$. 
			Applying what we just proved to all pairs $B_i, B_j$ shows that for any
			distinct $i, j \in \irange{k}$, $\mov{B_i}{d}$ is joined to $B_j - x_1$. 
			Put $F_1 = B_1$ and $F_i = B_i - x_1$ for $i \geq 2$.  
			Let $Q$ be the union of the $F_i$.  Then (a), (b) and (c) are satisfied.
			Note that $\card{Q} = \wt{\vec{r}}+1$ and since any $v \in B_1^d$ is universal in $Q$,
			$\card{Q} \leq d + 1$. By assumption $\wt{\vec{r}} \geq d$, whence $\wt{\vec{r}}=d$.  
			Hence, if (2) fails, then (1) holds.
\end{proof}

		\begin{thm}\label{DegenProp}
			Let $G$ be a graph, $k,d \in \IN$ with $k \geq 2$ and $\vec{r} \in \IN_{\geq 1}^k$ where at most one of the $r_i$ is one.  If $\wt{\vec{r}} \geq \max\set{\Delta(G) + 1 - k, d}$, then at least one of the following holds:
			\begin{enumerate}
			  \item $\wt{\vec{r}} = d$ and $G$ contains a $\join{K_t}{E_{d+1-t}}$ where $t \geq d +
			  1 - k$, for each $v \in V(K_t)$ we have $d_G(v) = d$ and for each $v \in
			  V(E_{d+1-t})$ we have $d_G(v) > d$; or,
			  \item there exists an $\vec{r}$-partition $P \DefinedAs \parens{V_1, \ldots, V_k}$ of 	
$G$ such that if $C$ is an $r_i$-regular component of $G[V_i]$, then $\delta_G(C) \geq d$ and
			  there is at most one $x \in V(\mov{C}{d})$ with $d_{\mov{C}{d}}(x) \geq
			  r_i -1$.  Moreover, $P$ can be chosen so that either: 
			  \begin{enumerate}
			    \item for all $i \in \irange{k}$ and $r_i$-regular component $C$ of
			    $G[V_i]$, we have $\card{\mov{C}{d}} \leq 1$; or,
			  	\item  for some $i \in \irange{k}$ and some $r_i$-regular component $C$ of
			  	$G[V_i]$, there is $x \in V(\mov{C}{d})$ such that
			  	$\setb{y}{N_C(x)}{d_G(y) = d}$ is a clique.
			  \end{enumerate}
			\end{enumerate}
		\end{thm}
		\begin{proof}
			For $i \in \irange{k}$, call a connected graph $C$
			\emph{$i$-bad} if $C$ is $r_i$-regular and there are at least two $x \in
			V(\mov{C}{d})$ with $d_{\mov{C}{d}}(x) \geq r_i - 1$. For a graph $H$ and $i
			\in \irange{k}$, let $b_i(H)$ be the number of $i$-bad components of $H$.
			For an $\vec{r}$-partition $P \DefinedAs \parens{V_1, \ldots, V_k}$ of $G$ let
			\[c(P) \DefinedAs \sum_{i \in \irange{k}} c(G[V_i]),\]
			\[b(P) \DefinedAs \sum_{i\in \irange{k}} b_i(G[V_i]).\]
			
			\noindent Let $P \DefinedAs \parens{V_1, \ldots, V_k}$ be an $\vec{r}$-partition of
			$V(G)$ minimizing $c(P)$ and subject to that $b(P)$.

			Let $i \in \irange{k}$ and $x \in V_i$ with $d_{V_i}(x) \geq r_i$.  Suppose $d_G(x) = d$.
			Then, since $\wt{\vec{r}} \geq d$, for every $j \neq i$ we have $d_{V_j}(x) \leq
			r_j$. Moving $x$ from $V_i$ to $V_j$ gives a new partition $P^*$ with $f(P^*)
			\leq f(P)$. Note that if $d_{G}(x) < d$ we would have $f(P^*) < f(P)$
			contradicting the minimality of $P$.

			Suppose $b(P) > 0$.  By symmetry, we may assume that there is a
			$1$-bad component $A_1$ of $G[V_{1}]$. Put $P_1 \DefinedAs P$ and $V_{1,i}
			\DefinedAs V_i$ for $i \in \irange{k}$. Since $A_1$ is $1$-bad we have $x_1
			\in V(\mov{A_1}{d})$ with $d_{\mov{A_1}{d}}(x) \geq r_1 -1$. By the above we
			can move $x_1$ from $V_{1, 1}$ to $V_{1, 2}$ to get a new partition $P_2
			\DefinedAs \parens{V_{2, 1}, V_{2,2}, \ldots, V_{2,k}}$ where $f(P_2) = f(P_1)$.  
			By the minimality of $c(P_1)$, $x_1$ is adjacent to only one component $C_2$
			in $G[V_{1, 2}]$. Let $A_2 \DefinedAs G[V(C_2) \cup \set{x_1}]$.  
			Since removing $x_1$ from $A_1$ decreased $b_{1}(G[V_{1}])$, minimality of
			$b(P_1)$ implies that $A_2$ is $2$-bad. Now, we may choose $x_2 \in
			V(\mov{A_2}{d}) - \set{x_1}$ with $d_{\mov{A_2}{d}}(x) \geq r_2 -1$ and move
			$x_2$ from $V_{2, 2}$ to $V_{2, 1}$ to get a new partition $P_3
			\DefinedAs \parens{V_{3, 1}, V_{3,2}, \ldots, V_{3,k}}$ where $f(P_3) =
			f(P_1)$.
						
			Continue on this way to construct sequences $A_1, A_2, \ldots$, $P_1, P_2, P_3, \ldots$ and $x_1, x_2, \ldots$.  
			Since $G$ is finite, at some point we will need to reuse a leftover component; that is, 
			there is a smallest $t$ such that $A_{t + 1} - x_t = A_s - x_s$ for some $s <
			t$.  Let $j \in \irange{2}$ be such that in $V(A_s) \subseteq V_{s, j}$. 
			Then $V(A_t) \subseteq V_{t, 3-j}$.  Note that, since $A_s$ is $r_j$-regular,
			$N(x_t) \cap V(A_s - x_s) = N(x_s) \cap V(A_s - x_s)$.
			
			We claim that $s = 1$, $t = 2$, both $A_s$ and $A_t$ are complete,
			$\mov{A_s}{d}$ is joined to $A_t - x_{t-1}$ and $\mov{A_t}{d}$ is joined to $A_s - x_s$.
			
			Put $X \DefinedAs N(x_s) \cap V(\mov{A_s}{d})$.  Since $x_s$ witnesses the
			$j$-badness of $A_s$, $\card{X} \geq \max\set{1, r_j - 1}$. Pick $z \in X$.
			In $P_s$, move $z$ to $V_{s, 3-j}$ to get a new partition $P^\gamma \DefinedAs \parens{V_{\gamma,
			1}, V_{\gamma, 2}, \ldots, V_{\gamma, k}}$. Then $z$ must create an
			$r_{3-j}$-regular component with $A_t - x_{t-1}$ in $V_{\gamma, 3-j}$ since
			$z$ is adjacent to $x_t$.  
			In particular, $N(z) \cap V(A_t - x_{t-1}) = N(x_{t-1}) \cap V(A_t -
			x_{t-1})$. Since $z$ is adjacent to $x_t$, so is $x_{t-1}$. 
			
			Suppose $r_j \geq 2$. In $P^\gamma$, move $x_t$ to $V_{\gamma, j}$ to
			get a new partition $P^{\gamma*} \DefinedAs \parens{V_{\gamma*, 1},
			V_{\gamma*, 2}, \ldots, V_{\gamma*, k}}$. Then $x_t$ must create an
			$r_{j}$-regular component with $A_s - z$ in $V_{\gamma*, j}$.  In
			particular, $N(z) \cap V(A_s - z) = N(x_t) \cap V(A_s - z)$.  Thus $x_s$
			is adjacent to $x_t$ and we have $N[z] \cap V(A_s) = N[x_s] \cap
			V(A_s)$. Put $K \DefinedAs X \cup \set{x_s}$.  Then $\card{K} \geq r_j$ and
			$K$ induces a clique.  If $\card{K} > r_j$, then $A_s = K$ is complete. 
			Otherwise, the vertices of $K$ have a common neighbor $y \in V(A_s) - K$ and
			again $A_s$ is complete. Also, since $x_s$ is adjacent to $x_t$, using $x_s$
			in place of $z$ in the previous paragraph, we conclude that $K$
			is joined to $N(x_{t-1}) \cap V(A_t - x_{t-1})$ and $x_s = x_{t-1}$.  
			
			Suppose $s > 1$.  Then $x_{s-1}$ is joined to $N(x_{t-1}) \cap V(A_t -
			x_{t-1})$ and hence $A_t - x_{t-1} = A_{s - 1} - x_{s-1}$ violating
			minimality of $t$.  Whence, if $r_j \geq 2$ then $s = 1$.  
			
			Note that $K = V(\mov{A_s}{d})$ and hence if $r_j \geq 2$ then $A_s$ is
			complete and $\mov{A_s}{d}$ is joined to $N(x_{t-1}) \cap V(A_t - x_{t-1})$.  If $r_{3-j}
			= 1$, then $A_t$ is a $K_2$ and $N(x_{t-1}) \cap V(A_t - x_{t-1}) = V(A_t -
			x_{t-1}) = \set{x_t}$.  We already know that $x_t$ is joined to $A_s - x_s$. 
			Thus the cases when $r_j \geq 2$ and $r_{3-j} = 1$ are taken care of. By
			assumption, at least one of $r_j$ or $r_{3-j}$ is at least two.  Hence it
			remains to handle the cases with $r_{3-j} \geq 2$.
			
			Suppose $r_{3-j} \geq 2$.  In $P_{t+1}$, move $z$ to $V_{t+1, 3-j}$ to
			get a new partition $P^{\beta} \DefinedAs \parens{V_{\beta, 1},
			V_{\beta, 2}, \ldots, V_{\beta, k}}$. Then $z$ must create an
			$r_{3-j}$-regular component with $A_t - x_t$ in $V_{\beta, 3-j}$.  In
			particular, $N(z) \cap V(A_t - x_t) = N(x_t) \cap V(A_t - x_t)$.  Since $N(z)
			\cap V(A_t - x_{t-1}) = N(x_{t-1}) \cap V(A_t - x_{t-1})$, we have
			$N[x_{t-1}] \cap V(A_t) = N(z) \cap V(A_t) = N[x_t] \cap V(A_t)$.  Put $W
			\DefinedAs N[x_t] \cap V(\mov{A_t}{d})$. Each $w \in W$ is adjacent to $z$
			and running through the argument above with $w$ in place of $x_t$ shows
			that $W$ is a clique joined to $z$.  Moreover, since $x_t$ witnesses the
			$(3-j)$-badness of $A_t$, $\card{W} \geq r_{3-j}$.  As with $A_s$ above, we
			conclude that $A_t$ is complete.  Since $x_s \in V_{t+1, 3-j}$ and $x_s$ is
			adjacent to $z$, it must be that $x_s \in V(A_t - x_t)$.  Thence $x_s$ is
			joined to $W$ and $x_s = x_{t-1}$.  
			
			Suppose that $r_j \geq 2$ as well.  We know that $s = 1$, $A_s$ is
			complete and $\mov{A_s}{d}$ is joined to $N(x_{t-1}) \cap V(A_t - x_{t-1})
			= A_t - x_{t-1}$.  Also, we just showed that $A_t$ is complete and
			$\mov{A_t}{d}$ is joined to $A_s - x_s$.
						
			Thus, we must have $r_j = 1$ and $r_{3-j} \geq 2$.  Then,
			since $A_s$ is a $K_2$, by the above, $A_s$ is joined to $W$.  Since $W =
			\mov{A_t}{d}$, it only remains to show that $s=1$. Suppose $s > 1$.  Then
			$x_{s-1}$ is joined to $W$ and hence $A_t - x_{t-1} = A_{s - 1} - x_{s-1}$ violating minimality of $t$.
			
			Therefore $s = 1$, $t = 2$, both $A_s$ and $A_t$ are complete,
			$\mov{A_s}{d}$ is joined to $A_t - x_{t-1}$ and $\mov{A_t}{d}$ is joined to
			$A_s - x_s$.  But we can play the same game with $V_1$ and $V_i$ for any
			$3 \leq i \leq k$ as well.  Let $B_1 \DefinedAs A_1$, $B_2 \DefinedAs A_2$
			and for $i \geq 3$, let $B_i$ be the $r_i$-regular component made by moving
			$x_1$ into $V_i$.  Then $B_i$ is complete for each $i \in \irange{k}$. 
			Applying what we just proved to all pairs $B_i, B_j$ shows that for any
			distinct $i, j \in \irange{k}$, $\mov{B_i}{d}$ is joined to $B_j - x_1$. 
			Since $\card{\mov{B_i}{d}} \geq r_i$ and $x_1 \in V(\mov{B_i}{d})$ for each
			$i$, this gives a $\join{K_t}{E_{\wt{\vec{r}} + 1 - t}}$ in $G$ where $t \geq \wt{\vec{r}} + 1 -
			k$.  Take such a subgraph $Q$ maximizing $t$.  Since all the $B_i$ are
			complete, any vertex of degree $d$ will be in $\mov{B_i}{d}$; therefore, for each $v
			\in V(K_t)$ we have $d_G(v) = d$ and for each $v \in V(E_{\wt{\vec{r}}+1-t})$ we have $d_G(v) > d$.
			Note that $\card{Q} = \wt{\vec{r}}+1$ and since $d_G(v) = d$ for any $v \in V(K_t)$,
			$\card{Q} \leq d + 1$. By assumption $\wt{\vec{r}} \geq d$, whence $\wt{\vec{r}}=d$.  
			Thus if (1) fails, then	the first part of (2) holds.
			
			It remains to prove that we can choose $P$ to satisfy one of (a) or (b).  Suppose that (1) fails and $P$ cannot be chosen to satisfy either (a) or (b).  For $i \in \irange{k}$, call a connected graph $C$
			\emph{$i$-ugly} if $C$ is $r_i$-regular and $\card{\mov{C}{d}} \geq 2$ let $u_i(H)$ be the number of $i$-ugly components of $H$.  Note that if $C$ is $i$-bad, then it is $i$-ugly.  For an $\vec{r}$-partition $P \DefinedAs \parens{V_1, \ldots, V_k}$ of $G$ let
\[u(P) \DefinedAs \sum_{i\in \irange{k}} u_i(G[V_i]).\]

Choose an $\vec{r}$-partition $Q \DefinedAs \parens{V_1, \ldots, V_k}$ of $G$ first minimizing $c(Q)$, then subject to that requiring $b(Q) \leq 1$ and then subject to that minimizing $u(Q)$.  Since $Q$ does not satisfy (a), at least one of $b(Q) = 1$ or $u(Q) \geq 1$ holds.  By symmetry, we may assume that $G[V_1]$ contains a component $D_1$ which is either $1$-bad or $1$-ugly (or both).  If $D_1$ is $1$-bad, pick $w_1 \in V(D_1^d)$ witnessing the $1$-badness of $D_1$; otherwise pick $w_1 \in V(D_1^d)$ arbitrarily. Move $w_1$ to $V_2$, to form a new $\vec{r}$-partition.  This new partition still satisfies all of our conditions on $Q$. As above we construct a sequence of vertex moves that will wrap around on itself. This can be defined recursively as follows.  For $t \geq 2$, if $D_t$ is bad pick $w_t \in V(D_t^d - w_{t-1})$ witnessing the badness of $D_t$; otherwise, if $D_t$ is ugly pick $w_t \in V(D_t^d - w_{t-1})$ arbitrarily.  Now move $w_t$ to the part from which $w_{t-1}$ came to form $D_{t+1}$.  Let $Q_1 \DefinedAs Q, Q_2, Q_3, \ldots$ be the partitions created by a run of this process. Note that the process can never create a component which is not ugly lest we violate the minimality of $u(Q)$.  

Since $G$ is finite, at some point we will need to reuse a leftover component; that is, 
there is a smallest $t$ such that $D_{t + 1} - x_t = D_s - x_s$ for some $s <
t$.  First, suppose $D_s$ is not bad, but merely ugly.  Then $D_{t+1}$ is not bad and hence $b(Q_{t+1}) = 0$ and $u(Q_{t+1}) < u(Q)$, a contradiction.  Hence $D_s$ is bad.  

Suppose $D_t$ is not bad.  
As in the proof of the first part of (2), we can conclude that $x_s = x_{t-1}$.  
Pick $z \in N(x_s) \cap V(\mov{D_s}{d})$. 
Since $z$ is adjacent to $x_t$, by moving $z$ to the part containing $x_t$ in $P_s$ we conclude 
$N(z) \cap V(D_t - x_s) = N(x_s) \cap V(D_t - x_s)$.  
Put $T \DefinedAs \setb{y}{N_{D_t}(x_s)}{d_G(y) = d}$. 
Suppose $T$ is not a clique and let $w_1, w_2 \in T$ be nonadjacent.  
Now, in $P_t$, since $z$ is adjacent to both $w_1$ and $w_2$, swapping $w_1$ and $w_2$ with $z$ contradicts minimality of $f(Q)$.  
Hence $T$ is a clique and (b) holds, a contradiction.

Thus we may assume that $D_t$ is bad as well.  
Now we may apply the same argument as in the proof of the first part of (2) to show that (1) holds.  This final contradiction completes the proof.

		\end{proof}

		\begin{cor}[Borodin \cite{borodin1976decomposition}]
		Let $G$ be a graph not containing a $K_{\Delta(G) + 1}$. If $r_1, r_2 \in
		\IN_{\geq 1}$ with $r_1 + r_2 \geq \Delta(G) \geq 3$, then $V(G)$ can be
		partitioned into sets $V_1, V_2$ such that $\Delta(G[V_i]) \leq r_i$ and $\text{col}(G[V_i]) \leq r_i$ for $i \in \irange{2}$.
		\end{cor}
		\begin{proof}
		Apply Proposition \ref{DegenProp} with $\vec{r} \DefinedAs \parens{r_1, r_2}$
		and $d = \Delta(G)$.  Since $G$ doesn't contain a $K_{\Delta(G) + 1}$ and no
		vertex in $G$ has degree larger than $d$, (1) cannot hold.  Thus (2) must
		hold.  Let $P \DefinedAs (V_1, V_2)$ be the guaranteed partition and suppose
		that for some $j \in \irange{2}$, $G[V_j]$ contains an $r_j$-regular component
		$H$.  Then every vertex of $H$ has degree $d$ in $G$ and
		hence $\mov{H}{d}$ contains all noncutvertices of $H$.  But $H$ has maximum
		degree $r_j$ and thus contains at least $r_j$ noncutvertices.  If $r_j = 1$, then $H$ is $K_2$ and hence has $2$ noncutvertices. In any case,
		we have $\card{\mov{H}{d}} \geq 2$.  Hence (a) cannot hold for $P$.  Thus, by (b),
		we have $i \in \irange{2}$, an $r_i$-regular component $C$ of $G[V_i]$ and $x
		\in V(C)$ such that $N_C(x)$ is a clique.  But then $C$ is $K_{r_i + 1}$
		violating (2), a contradiction.
		
		Therefore, for $i \in \irange{2}$, each component of $G[V_i]$ contains a
		vertex of degree at most $r_i - 1$.  Whence $\text{col}(G[V_i]) \leq r_i$ for
		$i \in \irange{2}$.
		\end{proof}

	\section{Coloring}
	Using Theorem \ref{PartitionTheorem}, we can prove coloring results for graphs
	with only small cliques among the vertices of high degree. To make this
	precise, for $d \in \IN$ define $\omega_d(G)$ to be the size of the largest
	clique in $G$ containing only vertices of degree larger than $d$; that is, $\omega_d(G)
	\DefinedAs \omega\parens{G\brackets{\setbs{v \in V(G)}{d_G(v) > d}}}$.
	
	\begin{cor}\label{FirstColoringCorollary}
	Let $G$ be a graph, $k,d \in \IN$ with $k \geq 2$ and $\vec{r} \in \IN^k$.  If
	$\wt{\vec{r}} \geq \max\set{\Delta(G) + 1 - k, d}$ and $r_i \geq \omega_d(G)
	+ 1$ for all $i \in \irange{k}$, then at least one of the following holds:
			\begin{enumerate}
           \item $\wt{\vec{r}} = d$ and $G$ contains an induced subgraph $Q$ with $\card{Q} = d+1$ which can be partitioned into $k$ cliques $F_1, \ldots, F_k$ where 
					\begin{enumerate}
					\item $\card{F_1} = r_1 + 1$, $\card{F_i} = r_i$ for $i \geq 2$,
					\item $\card{F_i^d} \geq \card{F_i} - \omega_d(G)$ for $i \in \irange{k}$,
					\item for $i \in \irange{k}$, each $v \in V(F_i^d)$ is universal in $Q$;
					\end{enumerate}
			  \item $\chi(G) \leq \wt{\vec{r}}$. 
			\end{enumerate}
	\end{cor}
	\begin{proof}
		Apply Theorem \ref{PartitionTheorem} to conclude that either (1) holds or there exists an $\vec{r}$-partition $P \DefinedAs \parens{V_1, \ldots, V_k}$ of 	
$G$ such that if $C$ is an $r_i$-obstruction in $G[V_i]$, then $\delta_G(C) \geq
d$ and $\mov{C}{d}$ is edgeless.  Since $\Delta(G[V_i]) \leq r_i$ for all $i
\in \irange{k}$, it will be enough to show that no $G[V_i]$ contains an
$r_i$-obstruction.  Suppose otherwise that we have an $r_i$-obstruction $C$ in
some $G[V_i]$.  First, if $r_i \geq 3$, then $C$ is $K_{r_i + 1}$ and hence $C$
contains a $K_{\omega_d(G) + 2}$.  But $\mov{C}{d}$ is edgeless, so
$\omega_d(G) > \omega_d(G) + 1$, a contradiction.  Thus $r_i = 2$ and $C$ is an
odd cycle.  Since $\mov{C}{d}$ is edgeless, the vertices of $C$ are $2$-colored
by the properties `degree is $d$' and `degree is greater than $d$',
impossible.
	\end{proof}

	For a vertex critical graph $G$, call $v \in V(G)$	$\emph{low}$ if $d(v) = \chi(G) - 1$ and $\emph{high}$ otherwise. 
	Let $\fancy{H}(G)$ be the subgraph of $G$ induced on the high vertices of $G$.
	
	\begin{cor}\label{SecondColoring}
	Let $G$ be a vertex critical graph with $\chi(G) = \Delta(G) + 2 - k$ for some
	$k \geq 2$.  If $k \leq \frac{\chi(G) - 1}{\omega(\fancy{H}(G)) + 1}$,
	then $G$ contains an induced subgraph $Q$ with $\card{Q} = \chi(G)$ which can be partitioned into $k$ cliques $F_1, \ldots, F_k$ where 
					\begin{enumerate}
					\item $\card{F_1} = \chi(G) - (k-1)(\omega(\fancy{H}(G)) + 1)$, $\card{F_i} = \omega(\fancy{H}(G)) + 1$ for $i \geq 2$;
					\item for each $i \in \irange{k}$, $F_i$ contains at least $\card{F_i} - \omega(\fancy{H}(G))$ low vertices which are all universal in $Q$.
					\end{enumerate}
	\end{cor}
	\begin{proof}
		Suppose $k \leq \frac{\chi(G) - 1}{\omega(\fancy{H}(G)) + 1}$.  Put
		$r_i \DefinedAs \omega(\fancy{H}(G)) + 1$ for $i \in \irange{k} - \set{1}$ and $r_1
		\DefinedAs \chi(G) - 1 - (k-1)(\omega(\fancy{H}(G)) + 1)$.  Set $\vec{r}
		\DefinedAs \parens{r_1, r_2, \ldots, r_k}$.  Then $\wt{\vec{r}} = \chi(G) - 1
		= \Delta(G) + 1 - k$.  Now applying Corollary \ref{FirstColoringCorollary}
		with $d \DefinedAs \chi(G) - 1$ proves the corollary.
	\end{proof}
	
	\begin{cor}\label{ThirdColoring}
	Let $G$ be a vertex critical graph with $\chi(G) \geq \Delta(G) + 1 - p \geq 4$
	for some $p \in \IN$.  If $\omega(\fancy{H}(G)) \leq \frac{\chi(G) + 1}{p + 1} - 2$,
	then $G = K_{\chi(G)}$ or $G = O_5$.
	\end{cor}
	\begin{proof}
	Suppose not and choose a counterexample $G$ minimizing $\card{G}$.  Put
	$\chi \DefinedAs \chi(G)$, $\Delta \DefinedAs \Delta(G)$ and $h \DefinedAs
	\omega(\fancy{H}(G))$. Then $p \geq 1$ and $h \geq 1$ by Brooks' theorem. Hence
	$\chi \geq 5$. By assumption, we have $h \leq \frac{\chi + 1}{p+1} - 2 =
	\frac{\chi - 2p - 1}{p + 1} \leq \frac{\chi - p - 2}{p + 1}$ since $p \geq 1$. 
	Thus $p + 1 \leq \frac{\chi - 1}{h + 1}$ and we may apply Corollary
	\ref{ThirdColoring} with $k \DefinedAs p + 1$ to get an induced subgraph $Q$ of
	$G$ with $\card{Q} = \chi$ which can be partitioned into $p + 1$ cliques $F_1,
	\ldots, F_{p + 1}$ where
			\begin{enumerate}
					\item $\card{F_1} = \chi - p(h + 1)$, $\card{F_i} = h	+ 1$ for $i \geq 2$;
					\item for each $i \in \irange{p+1}$, $F_i$ contains at least $\card{F_i} -
					h$ low vertices which are all universal in $Q$.
			\end{enumerate}
Let $T$ be the low vertices in $Q$, put $H \DefinedAs Q - T$ and $t
\DefinedAs \card{T}$.  Then $Q = \join{K_t}{H}$ and $t \geq \chi - p(h + 1) +
p(h + 1) - (p + 1)h = \chi - (p + 1)h$.  

Take any $(\chi - 1)$-coloring of $G-Q$ and let $L$ be the resulting list
assignment on $Q$.  Then $\card{L(v)} = d_Q(v)$ for each $v \in T$ and
$\card{L(v)} \geq d_Q(v) - p$ for each $v \in V(H)$.  Since $t \geq \chi - (p +
1)h \geq 2p + 1 \geq p + 1$, if there are nonadjacent $x,y \in V(H)$ and $c \in
L(x) \cap L(y)$, then we may color $x$ and $y$ both with $c$ and then greedily
complete the coloring to the rest of $H$ and then to all of $Q$, a
contradiction.  Hence any nonadjacent pair in $H$ have disjoint lists.

Let $I$ be a maximal independent set in $H$. If there is an induced $P_3$
in $H$ with ends in $I$, set $o_I \DefinedAs 1$, otherwise set $o_I
\DefinedAs 0$. Since each pair of vertices in $I$ have disjoint lists, we must
have

\begin{align*}
	\chi - 1 &\geq \sum_{v \in I} \card{L(v)} \\
	&\geq \sum_{v \in I} t + d_H(v) - p \\
	&= (t-p)\card{I} + \sum_{v \in I} d_H(v) \\
	&\geq (t-p)\card{I} + \card{H} - \card{I} + o_I \\
	&= (t - (p + 1))\card{I} + \chi - t + o_I. 
\end{align*}
Hence $\card{I} \leq \frac{t-1 - o_I}{t - (p + 1)} = 1 +
\frac{p-o_I}{t-(p+1)} \leq 1 + \frac{p-o_I}{2p + 1 - (p+1)} \leq 2$ as $t \geq
2p + 1$.  Since $G$ is not $K_\chi$, we must
have $\card{I} = 2$ and thus $t = 2p + 1$ and $o_I = 0$.  Thence $H$ is the
disjoint union of two complete subgraphs.  We then have $\frac{\chi - 2p - 1}{p
+ 1} \geq h \geq \frac{\card{H}}{2} = \frac{\chi - 2p - 1}{2}$.  Hence $p =
1$, $h = \frac{\chi - 3}{2}$ and $Q = \join{K_3}{2K_h}$.

Let $x,y \in V(H)$ be nonadjacent.  Then $d_Q(x) + d_Q(y) = \chi + 1$.  Let $A$
be the subgraph of $G$ induced on $V(G - Q) \cup \set{x,y}$.  Then $d_A(x) + d_A(y) \leq 2\Delta - (\chi + 1) = \chi - 1$.  Let
$A'$ be the graph obtained by collapsing $\set{x, y}$ to a single vertex
$v_{xy}$. If $\chi(A') \leq \chi - 1$, then we have a $(\chi - 1)$-coloring of
$A$ in which $x$ and $y$ receive the same color.  This is impossible as then we could
complete the $(\chi - 1)$-coloring to all of $G$ greedily as above.  Hence
$\chi(A') = \chi$ and thus we have a vertex critical subgraph $Z$ of $A'$ with
$\chi(Z) = \chi$.  We must have $v_{xy} \in V(Z)$ and since $d_A(x) + d_A(y)
\leq \chi - 1$, $v_{xy}$ is low.  Hence, by minimality of $\card{G}$, $Z =
K_\chi$ or $Z = O_5$.

First, suppose $\chi \geq 6$.  Then $h \geq 2$ and thus we have $z \in V(H) - \set{x, y}$ nonadjacent to $x$.  
Apply the previous paragraph to both pairs $\set{x, y}$ and $\set{x, z}$.  
The case $Z = O_5$ cannot happen, for then we would have $\chi = \chi(Z) = 5$, a contradiction.  
Put $X_1 \DefinedAs N(x) \cap V(G - Q)$, $X_2 \DefinedAs N(y) \cap V(G - Q)$, $X_3 \DefinedAs N(z) \cap V(G - Q)$.  
Then $\card{X_i} = \frac{\chi - 1}{2}$ for $i \in \irange{3}$ and $X_1$ is joined to both $X_2$ and $X_3$.  
Since $\card{X_i} - h > 0$, each $X_i$ contains a low vertex $v_i$.  But then
$N(v_1) = X_1 \cup X_2 \cup \set{x}$ and we must have $X_3 = X_2$. Whence $N(v_2) = X_1 \cup X_2 \cup \set{y, z}$ giving $d(v_2) \geq \chi$, a contradiction.

Therefore $\chi = 5$, $h = 1$ and $V(H) = \set{x, y}$.  If $Z = K_5$, then $N[x] \cup N[y]$ induces an $O_5$ in $G$ and hence $G = O_5$, a contradiction.  
Thus $Z = O_5$.  But $h = 1$, so all of the neighbors of both $x$ and $y$ are
low and hence all of the neighbors of $v_{xy}$ in $Z$ are low. But $O_5$ has no such low vertex $v_{xy}$ with all low neighbors, so this is impossible.
	\end{proof}

\begin{question}
The condition on $k$ needed in Corollary \ref{SecondColoring} is weaker than that in Corollary \ref{ThirdColoring}.  What do the intermediate cases look like?  What are the extremal examples?
\end{question}

\section{Acknowledgements}
\noindent Thanks to Hal Kierstead for many helpful discussions of this material.

\bibliographystyle{amsplain}
\bibliography{GraphColoring}
\end{document}

%% file: reducer.tex
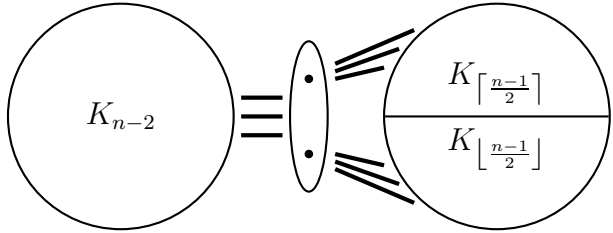
\begin{figure}[ht]
\centering
\begin{tikzpicture}[scale = 1]

\node[circle, minimum width=3cm, thick, draw] (L) at (0,0) {$K_{n-2}$};
\node[ellipse, minimum height=2cm, minimum width=0.5cm, thick, draw] (M) at (2.5,0) {};      
\node[circle split, minimum width=3cm, thick, draw] (R) at (5,0) 
{$K_{\ceil{\frac{n-1}{2}}}$ \nodepart{lower} $K_{\floor{\frac{n-1}{2}}}$};
\node[circle, inner sep =1pt, fill, draw] (P1) at (2.5,-0.5) {};
\node[circle, inner sep =1pt, fill, draw] (P2) at (2.5,0.5) {};

\draw (L) (M) (R) (P1) (P2);
\draw[ultra thick] (1.6,-.25) -- (2.15,-.25);
\draw[ultra thick] (1.6,0) -- (2.15,0);
\draw[ultra thick] (1.6,.25) -- (2.15,.25);

\draw[ultra thick] (2.85,-.7) -- (3.9,-1.15);
\draw[ultra thick] (2.85,-.6) -- (3.7,-.9);
\draw[ultra thick] (2.85,-.5) -- (3.5,-.65);

\draw[ultra thick] (2.85,.7) -- (3.9,1.15);
\draw[ultra thick] (2.85,.6) -- (3.7,.9);
\draw[ultra thick] (2.85,.5) -- (3.5,.65);
\end{tikzpicture}
\caption{The graph $O_n$.}
\label{fig:reducer}
\end{figure}